\documentclass{article}

\usepackage{amsmath,amsfonts,amsthm,amssymb,amscd,color,xcolor,mathrsfs,verbatim,microtype}
\usepackage{graphicx,eurosym}
\usepackage{hyperref}

\usepackage[applemac]{inputenc}

\usepackage[cyr]{aeguill}

\colorlet{darkblue}{blue!50!black}

\hypersetup{
    colorlinks,%
    citecolor=blue,%
    filecolor=red,%
    linkcolor=darkblue,%
    urlcolor=blue,%
    pdfnewwindow=true,%
    pdfstartview={FitH}
}

\usepackage{graphicx,amscd,mathrsfs,wrapfig,mathrsfs,lipsum}
\usepackage{eufrak}
\usepackage{float}
\usepackage{tikz}
\usepackage{multicol}
\usepackage{caption}
\usetikzlibrary{arrows}
\usepackage{capt-of}

\colorlet{darkblue}{blue!50!black}

\newcommand{\p}{\partial}
\newcommand{\e}{\varepsilon}

\newcommand{\D}{{\mathbb D}}

\newcommand{\R}{{\mathbb R}}
\newcommand{\Z}{{\mathbb Z}}

\newcommand{\T}{{\mathbb T}}

\newcommand{\ty}{\infty}

\newcommand{\de}{\delta}

\newcommand{\EE}{{\cal E}}

\newcommand{\GG}{{\cal G}}
\newcommand{\HH}{{\cal H}}

\newcommand{\KK}{{\cal K}}

\newcommand{\RR}{S}

\newcommand{\XX}{{\cal X}}

\newcommand{\lag}{\langle}
\newcommand{\rag}{\rangle}

\newcommand{\dd}{{\textup d}}

\newcommand{\lspan}{\mathop{\rm span}\nolimits}

\newcommand{\diver}{\mathop{\rm div}\nolimits}

\theoremstyle{plain}
\newtheorem*{mt}{Main Theorem}

\newtheorem*{lemma*}{Lemma}
\newtheorem{theorem}{Theorem}[section]

\newtheorem{proposition}[theorem]{Proposition}
\newtheorem{corollary}[theorem]{Corollary}
\newtheorem*{corollary*}{Corollary}
\theoremstyle{definition}
\newtheorem{definition}[theorem]{Definition}

\theoremstyle{remark}

\newtheorem{remark}[theorem]{Remark}

\numberwithin{equation}{section}

\begin{document}
\author{Vahagn~Nersesyan\footnote{Universit\'e Paris-Saclay, UVSQ, CNRS, Laboratoire de Math\'ematiques de Versailles, 78000, Versailles, France;  e-mail: \href{mailto:vahagn.nersesyan@uvsq.fr}{Vahagn.Nersesyan@uvsq.fr}}}
  \title{A   proof of approximate controllability of the 3D~Navier--Stokes system    via a linear test}
\date{\today}

\maketitle

 \begin{abstract}

 We consider  the   3D~Navier--Stokes system driven by an   additive finite-dimensional control force.~The purpose  of this paper is to show how  the   approximate controllability     of this system  can be derived from the approximate controllability of   the Euler system linearised around some suitable trajectory.~The  proof presented here is        shorter than   the previous~ones obtained by  Lie algebraic methods and gives some   new information about the structure of the control.~The dimension of the   control space   provided by this approach is   larger, but  it is still uniform with respect to~the~viscosity.



\end{abstract}

 \tableofcontents
 
\setcounter{section}{-1}

\section{Introduction}
\label{S:0}

 In this paper, 
 we consider the 3D  Navier--Stokes (NS) system  for incompressible viscous fluids on the torus $\T^3=\R^3/2\pi\Z^3$:   \begin{equation} \label{0.1}
\p_t u-\nu\Delta u+\lag u,\nabla\rag u+\nabla p=f(t,x),\quad\diver u=0,  
\end{equation}
where $\nu>0$ is the viscosity of the fluid, $u = (u_1(t,x), u_2(t,x),u_3(t,x))$  and $p = p(t, x)$ are the  unknown velocity field and   pressure, and $f$ is an external~force. We   fix  any   $T>0$ and
   assume that the  force    is      of the~form 
   $$f(t,x)=h(t,x)+ \eta(t,x), \quad t\in J_T=[0,T], \quad x\in\T^3,$$ 
   where  $  h: J_T\times \T^3\to \R^3$ is a given smooth function   
and~$\eta$ is a control taking values in some   subspace $\HH \subset H^k(\T^3,\R^3)$, $k\ge 3$.~The subspace $\HH$  incorporates different constraints  that might be imposed on the control;   in the examples   considered in this paper,  it   gives the Fourier modes that are directly perturbed by the   control force. We are mostly interested here by the situation when $\HH$ is   a  finite-dimensional subspace     not depending~on~the~viscosity.

Projecting Eq.~\eqref{0.1} to the space~$H$ of divergence-free vector fields with zero mean~value (see~\eqref{0.9}) and assuming that $h(t)$ and $\eta(t)$ belong to $H$ for any~$t\in J_T$,  
  we~eliminate the pressure term from the problem and obtain an evolution equation for  the velocity field:
\begin{equation} \label{0.2}
\dot u+\nu L  u  +B(u)=h + \eta.  
\end{equation}This equation is supplemented with the initial condition
 \begin{equation}\label{0.3}
u(0)=u_0	.
\end{equation} 
 Recall that, for any~$u_0\in H^k=H^k(\T^3,\R^3)\cap H$, problem~\eqref{0.2},~\eqref{0.3} has a unique local-in-time strong solution (see Section~\ref{S:1.1}).

 In this introduction, we formulate a simplified version of   our main  result, assuming that   the   subspace $\HH$ is   given by
\begin{equation}\label{0.4}
\HH=\lspan\{l(\ell)\sin\lag \ell,x\rag, l(\ell)\cos\lag \ell,x\rag:~|\ell|\le 2,~\ell \in \Z_*^3\}, 
\end{equation}
where $\{l(\ell), l(-\ell)\}$ is   an arbitrary orthonormal basis in~$\{x \in  \R^3:  \lag \ell, x\rag= 0\}.$  
 \begin{mt}Eq.~\eqref{0.2} is approximately controllable in   small time by $\HH$-valued controls, i.e.,~for any initial condition $u_0\in H^{k+1}$, any target $u_1\in H^{k+1}$, and sufficiently small    $\delta>0$, there is a   control   $\eta_\delta\in  L^2(J_{T\delta},\HH)$ and a   strong solution~$u$ of   problem~\eqref{0.2},~\eqref{0.3}  defined on $J_{T\de}$ such that      
\begin{equation}\label{0.5}
u(T\de)\to u_1 \quad\text{in $H^k$ as $\de\to 0^+$}.
\end{equation}Moreover, the control $\eta_\delta$ can be chosen in the form 
\begin{equation}\label{0.6}
  \eta_\de=R_\de(u_0,u_1)+ \zeta_\delta,
\end{equation}
 where $R_\de: H^k\times H^k\to  L^2(J_{T\delta},\HH)$ is a   linear bounded operator  with a finite-dimensional range and    $\zeta_\delta\in    L^2(J_{T\delta},\HH)$, both   $R_\de$ and $\zeta_\delta$ do   not depend  on~$(u_0,u_1)$. Limit~\eqref{0.5} is
   uniform with respect to $u_0$ and $u_1$ in a
bounded set of $H^{k+1}$.  \end{mt} 
A more general version of this result is given in Section~\ref{S:2}.~In particular, we~define    there a    saturation property that implies small time approximate controllability for different  subspaces $\HH$ spanned by  eigenfunctions of the Stokes operator.~As~a consequence of the Main Theorem, we obtain the following   approximate controllability property in fixed~time. 
\begin{corollary*}Eq.~\eqref{0.2} is approximately controllable in time $T>0$ by $\HH$-valued controls, i.e., 
	 for any~$\e>0$ and any $u_0, u_1\in H^k$, there is a control   $\eta\in L^2(J_T,\HH)$ and a   strong solution~$u$   of Eq.~\eqref{0.2} defined on $J_T$ such that     
$$
 \|u(T)-u_1\|_{H^k}<\e. 
$$
\end{corollary*}Roughly speaking, this result  is obtained by applying  the   Main Theorem on a~small time interval, then by forcing the trajectory to remain   near $u_1$  for  sufficiently~long~time.

The problem of   controllability of PDEs with  an additive finite-dimensional   force has been studied by many authors in the recent years.~Agrachev and Sarychev~\cite{AS-2005, AS-2006, AS-2008} were the first who considered this problem; they   established the approximate controllability of the~NS  and Euler systems on the 2D torus. Shirikyan generalised their approach to study      the     NS     system  on the 3D torus~\cite{shirikyan-cmp2006, shirikyan-aihp2007} and   the Burgers equation     on the real line~\cite{Shi-2013} and    on a bounded  interval with   Dirichlet boundary conditions~\cite{Shir-2018}. 
Rodrigues and Phan~\cite{SSRodrig-06,RD-2018} considered    the 2D and 3D NS systems  on    rectangles with   Lions boundary conditions. Compressible and incompressible 3D Euler systems were studied by Nersisyan~\cite{Hayk-2010,   Hayk-2011}, and the  2D~cubic Schr\"odinger equation by Sarychev~\cite{Sar-2012}.~More recently, the author considered the approximate controllability of      Lagrangian trajectories of the 3D~NS system~\cite{VN-2015} and parabolic PDEs with polynomially growing nonlinearities~\cite{nersesyan-2020}.~Boulvard et al.~\cite{BPN-2020} considered the 3D system of primitive equations   of meteorology and oceanology with control acting directly   only on the temperature equation.~The proofs of these papers are based on   infinite-dimensional  extensions of   Lie algebraic   methods. Most of them  provide    sharp results,   in the sense that they give necessary and sufficient conditions on the Fourier modes that should be perturbed by the control  in order to   ensure approximate~controllability.  

In this paper,  we take a different route.~We proceed by developing an   approach by  Coron~\cite{Cor-96}, who  considered the      approximate controllability of the  2D~NS system with   Navier slip boundary conditions and used  control forces that are localised in the physical space or on the  boundary.~That approach, called   return method, has been later used by~Coron~and Fursikov~\cite{CF-96} to study the global exact controllability to trajectories of the 2D~NS system on   manifolds without boundary,  by Coron and  Glass~\cite{Coron-1996, Glass-00}  to consider the global exact boundary controllability of  the 2D and 3D Euler systems, by
Fursikov and  Imanuilov~\cite{FI-99} to study the global exact controllability to trajectories of the~3D~Boussinesq system, and by many other authors.  We refer the reader to the   Chapter~6 of the book~\cite{coron2007} for a detailed   discussion of the     return method, for    applications to different control problems, and for more references.

The present paper is the first to extend this method to the case of  forces that are localised in the Fourier space.~The   configuration  we use here does not provide sharp results\footnote{A sharp version of Corollary is obtained in the papers~\cite{shirikyan-cmp2006, shirikyan-aihp2007, VN-2015}.   The results of these papers imply, in particular, the  approximate controllability in fixed time $T>0$ by controls taking values in the smaller subspace 
 $
\lspan\{l(\ell)\sin\lag \ell,x\rag, l(\ell)\cos\lag \ell,x\rag:\,|\ell|\le 1,\,\ell \in \Z_*^3\}; 
$ see~Remark~\ref{R:2.4} for more details.
 
}   in terms of the number of Fourier modes  directly perturbed by the control, but gives new and simpler proof with new information about the structure of the control.~Roughly speaking, the idea of the proof consists in   developing $u(t)$ as~follows:
 $$
 u(t )=\de^{-1}w(\de^{-1}t)+v(\de^{-1}t)+r_\de(t) \quad\text{for small~$\delta>0$},
 $$ where $w(t)$ is a suitable   solution of the Euler system (cf.~\eqref{1.5}) and~$v(t)$ is a solution of the Euler system linearised around $w(t)$ (cf.~\eqref{1.6}); both correspond to some controls taking values in the subspace~$\HH$ defined by~\eqref{0.4}.
  We~take~$w(t)$ in the   form
\begin{equation}\label{0.7}
 w(t)=\sum_{\ell \in \Z^3_*, |\ell|\le 1} \left(\psi^c_{\ell}(t)  l(\ell)\cos\lag \ell, x\rag+\psi^s_{\ell}(t)  l(\ell)\sin\lag \ell, x\rag\right),
\end{equation}where the functions $\{\psi^c_{\ell}, \psi^s_{\ell}\}\subset W^{1,2}(J_T,\R)$   are chosen such that the boundary conditions
\begin{equation}\label{0.8}
\psi^c_{\ell}(0)=\psi^c_{\ell}(T)=\psi^s_{\ell}(0)=\psi^s_{\ell}(T)=0	
\end{equation}are satisfied 
  and
 the derivatives $\{\dot\psi^c_{\ell}, \dot\psi^s_{\ell}\}$    form an observable family. Replacing the expression~\eqref{0.7} of the function~$w(t)$  into the left hand side of the Euler system, we infer   that~$w(t)$ is indeed a solution corresponding to some    $\HH$-valued control.~The observability property implies that the linearised Euler system   is approximately controllable by $\HH$-valued controls. Furthermore, choosing~$\eta$ in the form~\eqref{0.6}, we show that~$\sup_{t\in J_{T\de}}\|r_\de(t)\|_{H^k} \to 0$ as~$\delta\to 0$.~In view of \eqref{0.8}, this implies that $u(t)$ behaves like $v(t)$ at the endpoints~$0$ and~$T\delta$ as~$\delta\to 0$. Then the approximate controllability of the linearised Euler equation allows to conclude~\eqref{0.5}. The operator~$R_\delta$ in~\eqref{0.6} is an approximate right inverse of the resolving operator of the linearised Euler system and $\zeta_\de$ is explicitly given in terms of   the solution~$w$ and the corresponding control.

 The   proof of the Main Theorem is general enough and  can be applied to many other equations, such as the complex Ginzburg--Landau equation, the~Euler system, and  parabolic PDEs with polynomial nonlinearities.

This paper is organised as follows. In Section~\ref{S:1}, we formulate a perturbative result on solvability of  the 3D NS system and explain under what conditions its approximate controllability can be derived from that of the linearised Euler system. In Section~\ref{S:2}, we show that the conditions in Section~\ref{S:1} are satisfied when a saturation property holds for the set of controlled Fourier modes.  Finally, in
  Section~\ref{S:3}, we discuss the validity of the saturation property.

  \subsubsection*{Acknowledgement}

  This~research   was supported by the   ANR through  the grant NONSTOPS   ANR-17-CE40-0006-02.

  \subsection*{Notation} 
 
Here we collect some   notation used in this paper. 
 
\smallskip
\noindent
$\Z^3$ is   the integer lattice in~$\R^3$, $\Z^3_*=\Z^3\backslash\{0\}$, and  $\T^3$ is the   torus~$\R^3/2\pi\Z^3$.

\smallskip
\noindent
$L^p(\T^3,\R^3),$ $1\le p<\ty$  is the   Lebesgue space 
     endowed with the   norm~$\|\cdot\|_{L^p}$.  
 
\smallskip
\noindent
$H^k(\T^3,\R^3)$ is the Sobolev space of order $k\ge 1$ endowed with the scalar product $(\cdot,\cdot)_k$ and the corresponding   norm~$\|\cdot\|_k$.  

\smallskip
\noindent
$H^k= H^k(\T^3, \R^3)\cap H$, where  \begin{equation}\label{0.9}
H= \left\{u\in L^2(\T^3,\R^3):\,\, \diver u=0 \,\, \text{in $\T^3$,} \,\,\int_{\T^3} u(x) \dd x=0   \right\}.
\end{equation} $H$ is endowed with the $L^2$ scalar product $\lag \cdot, \cdot\rag$ and the corresponding norm $\|\cdot\|.$

\smallskip

Let  $X$ be a Banach  space   with the
norm $\|\cdot\|_X$. Then $B_X(a,R)$ denotes   the closed ball in $X$ of radius $R>0$ centred at $a\in X$.

\smallskip
\noindent
$C(J_T,X)$ is the space of continuous functions $u:J_T=[0,T]\to X$ endowed with the norm 
$$
\|u\|_{C(J_T,X)}=\max_{t\in J_T} \|u(t)\|_X.
$$

\smallskip
\noindent
$L^p(J_T,X),$  $1\leq p<\infty$   is      the
space of measurable functions $u: J_T\to  X$ with the~norm
\begin{equation}
\|u\|_{L^p(J_T,X)}=\bigg(\int_{0}^T \|u(t)\|_X^p\dd t
\bigg)^{1/p}.\nonumber
\end{equation}

\noindent
$L^p_{\text{loc}}(\R_+,X)$ is the space of measurable functions $u: \R_+ \rightarrow X$  whose restriction to   $J_T$ belongs to  $L^p(J_T,X)$ for any  $T>0$.

\smallskip
\noindent
$W^{m,p}(J_T,X)$, $m\ge1$ is the space of functions $u:J_T\to \R $ such that $ \frac{\dd^i }{\dd t^i} u\in L^p(J_T,X)$ for $0\le i\le m$.

\smallskip
\noindent
Throughout this paper, the same letter $C$ is used to denote  unessential positive constants   that may change from line to line.

\section{Linear test for approximate controllability}\label{S:1}

\subsection{Perturbative  result}\label{S:1.1}

Projecting the NS   system   to the space $H$, we  rewrite it  in the following equivalent~form without pressure term
 \begin{align} 
\dot u+\nu L  u  +B(u)&=f,  \label{1.1}\\
u(0)&=u_0, \label{1.2}
\end{align}where $L=- \Delta$ is the Stokes operator,   $B(u)=\Pi(\lag u,\nabla \rag u),$ and   $\Pi$ is the Leray   orthogonal projection  onto  $H$ in~$L^2$. In this section, we recall  a  perturbative result for   problem~\eqref{1.1},~\eqref{1.2}. Let us take  any integer~$k\ge3$ and define the~space      
$$
\XX_{T,k}=C(J_T, H^k)\cap L^2(J_T,H^{k+1}) 
$$
     endowed   with the~norm 
$$
\|u\|_{\XX_{T,k}}=\|u\|_{C(J_T,H^k)}+\|u\|_{L^2(J_T,H^{k+1})}.
$$
  \begin{proposition}\label{P:1.1} Let $\hat u_0\in H^k$ and   $\hat f\in L^2_{loc}(\R_+,H^{k-1})$.~There is a maximal time~$T_*=T_*(\hat u_0,   \hat f)>0$ and a unique solution $\hat u$ of problem~\eqref{1.1},~\eqref{1.2} with $u_0 =\hat u_0 $ and $ f=  \hat f$     whose restriction to the interval $J_T$  belongs to~$\XX_{T,k}$ for any~$T<T_*$.~If~$T_*<\ty$, then   
$\|\hat u(t)\|_k\to +\infty$ as  $t\to T_*^-$.
  For any~$T<T_*$,  there are numbers
 $\varkappa=\varkappa(T,\Lambda)>0$  and $C=C(T,\Lambda)>0$, where   
$$
\Lambda=  \|\hat u\|_{\XX_{T,k}}+\|\hat f \|_{L^2(J_T,H^{k-1})},
$$
such that  
\begin{itemize}
\item[(i)] for any $u_0\in H^k$   and $f\in L^2(J_T,H^{k-1})$  satisfying  
\begin{equation}\label{1.3}
\|u_0-\hat u_0\|_k+ \|f-\hat f\|_{L^2(J_T,H^{k-1})}<
\varkappa,
\end{equation}
 there is  a unique
solution $ u\in \XX_{T,k}$ of problem \eqref{1.1}, \eqref{1.2};
\item[(ii)] 
  let $S$ be the resolving operator of problem \eqref{1.1}, \eqref{1.2}, i.e., the mapping  taking   $(u_0,f )$ satisfying \eqref{1.3} to the solution~$u$.
 Then     
$$\|S(u_0,f)-S(\hat u_0,\hat f)\|_{\XX_{T,k}}\le 
C \left(\|u_0-\hat u_0\|_k+
\|f-\hat f\|_{L^2(J_T,H^{k-1})}\right).
$$
 \end{itemize}  
   Moreover, the problem  is  regularising   in the sense that, when $\hat f$~is~smooth,  the restriction of $S(\hat u_0,\hat f)$   to the interval $(0, T_*]$ belongs to $C^\ty((0,T_*]\times \T^3)$. 
\end{proposition}See, for exemple,    Chapter~17 in~\cite{taylor1996} for   the  proof  of
  local  well-posedness  and   regularising property.~The properties (i) and~(ii), under these regularity assumptions, are proved in Theorem~1.3 in~\cite{VN-2015}, using   standard arguments.

In what follows, we fix any time $T>0$, any integer~$k\ge 3$, and any function $h\in L^2(J_T,H^{k-1})$, and assume that $f=h+\eta$. Let us~denote by~$\Theta(u_0,h, T)$ 
 the set of functions $\eta\in L^2(J_T,H^{k-1})$   such that problem~\eqref{1.1},~\eqref{1.2}  has
a solution $u\in \XX_{T,k}$.~In~view of   Proposition~\ref{P:1.1}, the set
  $\Theta(u_0,h, T)$ is   open in~$L^2(J_T,H^{k-1})$.~We denote by~$S_t(u_0,h+\eta)$ the   
   restriction  of the solution  at time $t<T_*(u_0,    h+\eta)$. 

\subsection{Formulation and proof}\label{S:1.2}
  
By developing   the arguments of~\cite{Cor-96}, we show in this section how    the   approximate controllability of the NS~system 
\begin{equation} \label{1.4}
\dot u+\nu L  u  +B(u)=h + \eta
\end{equation}can be derived from the approximate controllability of the linearised Euler system.
More precisely,  we consider the Euler system 
 \begin{equation}\label{1.5}
\dot w  + B(w)=\zeta, 
\end{equation}and   its linearisation   
 \begin{equation}\label{1.6}
\dot v +  Q(v,w)=g, 
\end{equation} where 
\begin{equation}\label{1.7}
	Q(v,w)=B(v,w)+B(w,v), \quad   B(v,w)=\Pi(\lag v,\nabla \rag w). 
\end{equation}
The functions $\eta, \zeta,$ and $g$ are considered as controls taking values in the same (finite or infinite-dimensional) subspace~$\HH$  of $H^{k+1}$.
We will use the following two conditions.
\begin{itemize}
\item[\hypertarget{C1}{(C$_1$)}] 
{\sl There is a function $\zeta \in L^2 (J_T, \HH)$ and a solution 
$
w\in   C(J_T, H^{k+2})\cap W^{1,2}(J_T,H^{k+1})
$ of Eq.~\eqref{1.5} such that}
\begin{gather}
	w(0)= w(T)=0, \label{1.8}\\
	 Lw(t)\in \HH\quad \text{for   $t\in J_T$}. \label{1.9}
\end{gather}

\item[\hypertarget{C2}{(C$_2$)}] 
{\sl The linear Eq.~\eqref{1.6}, with a reference  trajectory   $w$  as in  Condition~{\rm(\hyperlink{C1}{\rm C$_1$})}, is approximately controllable in time $T>0$, i.e., for any~$\e>0$ and   any~$v_1\in H^{k+1}$, there is   a control $g\in L^2(J_T,\HH)$ such that the   solution
$
v\in C(J_T, H^{k+1})\cap W^{1,2}(J_T,H^k)
$ of Eq.~\eqref{1.6} with initial condition $v(0)=0$ satisfies
$$
 \|v(T)-v_1\|_{k+1}<\e. 
$$
}
\end{itemize}
  \begin{proposition}\label{P:1.2}
Let $\HH$ be a   subspace of $H^{k+1}$ such that  Condition~{\rm(\hyperlink{C1}{\rm C$_1$})} is satisfied.
 Then for any $u_0\in H^{k+1}$,   any $g\in L^2(J_T,\HH)$,  
 and sufficiently small $\delta>0$, there is a   control   $\eta_\delta\in \Theta(u_0,h, T\delta) \cap  L^2(J_{T\delta},\HH)$ such that    
\begin{equation}\label{1.10}
S_{T\delta}(u_0,h+\eta_\delta)\to v(T) \quad\text{in $H^k$ as $\de\to 0^+$},
\end{equation}where $v\in C(J_T, H^{k+1})\cap W^{1,2}(J_T,H^{k})$ is the solution of Eq.~\eqref{1.6} with initial condition $v(0)=u_0$.
 Moreover,    $\eta_\de$ is given explicitly by   
\begin{equation}\label{1.11}
  \eta_\de=  \delta^{-1} g(\delta^{-1}t)+ \delta^{-2} \zeta(\delta^{-1}t)+\nu  \delta^{-1}L w(\delta^{-1}t), \quad t\in J_{T\delta},
\end{equation}and  limit~\eqref{1.10} is uniform with respect to $u_0$ in a
bounded set of $H^{k+1}$.  
 \end{proposition}
\begin{proof}  {\it Step~1.~Preliminaries.} Let us take any $M>0$, any $u_0\in B_{H^{k+1}}(0,M)$, and any
  $\eta\in L^2_{loc}(\R_+,\HH)$ and denote by~$u(t)=S_t(u_0,h+\eta)$, $t<T_*=T_*(u_0, h+\eta)$ the solution of problem~\eqref{1.4},~\eqref{1.2}. Following~\cite{Cor-96}, we   make a time substitution and consider the~functions
   \begin{gather}
   v_\de(t) =v(\delta^{-1}t),\quad   g_\de(t) =\delta^{-1} g(\delta^{-1}t),\nonumber\\
   w_\de(t) =\delta^{-1} w(\delta^{-1}t),\quad   \zeta_\de(t) =\delta^{-2} \zeta(\delta^{-1}t),\nonumber\\
   	r(t)=u(t)-v_\de(t)-w_\de(t), \quad t < \tilde T^\delta =\min\{T\delta, T_*\}.\label{1.12}
   \end{gather}   
  Assume that we have found a control $\eta=\eta_\delta\in L^2(J_T,\HH)$ such that
   \begin{equation}\label{1.13}
   T\delta< 	T_*^\delta =T_*(u_0, h+\eta_\de) \quad\text{for small $\delta>0$}.
   \end{equation}Then, in view of the equalities \eqref{1.8}, \eqref{1.12}, and $v(0)=u_0$, we   have
\begin{equation}\label{1.14}
r(0)=0, \quad   r(T\de)=u(T\de)-v(T).
\end{equation}
Thus, we  need to choose $\eta_\de\in L^2(J_T,\HH)$ such that, in addition to \eqref{1.13}, also the following limit holds 
\begin{equation}\label{1.15}
	\|r(T\de)\|_k\to 0 \quad \text{as $\delta\to 0^+$}, 
\end{equation} uniformly with respect to~$u_0\in B_{H^{k+1}}(0,M)$.
This will imply limit \eqref{1.10}.

  {\it Step~2.~Proof of \eqref{1.13} and \eqref{1.15}.} The functions $v_\delta(t)$ and $w_\delta(t)$, $t\in J_T$  satisfy the equations
  \begin{align*}
 \dot v_\de +Q(v_\de,w_\de)&=g_\delta,  \\    
 \dot w_\de +Q(w_\de)&=\zeta_\delta.
  \end{align*}This implies that~$r$ is a solution of the~equation 
\begin{equation}\label{1.16}
\dot r+ \nu  Lr+ B(r,r+v_\de+w_\delta)+	 B(v_\de+w_\delta, r)=  \xi_\de, \quad t<\tilde T^\delta,
\end{equation}where 
$$
\xi_\de= h+ \eta_\de- \nu  Lv_\delta-\nu  Lw_\delta-B(v_\de)-g_\de-\zeta_\de.
$$
Choosing  $\eta_\delta=g_\de+ \zeta_\de+\nu  Lw_\delta\in L^2(J_{T\delta},\HH)$ (cf. \eqref{1.11}),  we get
\begin{equation}\label{1.17}
\xi_\de= h-  \nu Lv_\delta-B(v_\de).
\end{equation}
Taking  the scalar product in $H$ of Eq.~\eqref{1.16} with $L^k r$, integrating by parts, then integrating in time, and using the first equality in~\eqref{1.14}, we obtain
\begin{align}\label{1.18}
\frac12  \|r\|_k^2+ \nu \int_0^t\|r\|_{k+1}^2 \dd s&=	\int_0^t \lag \xi_\delta, L^k r\rag\, \dd s-\int_0^t \lag   B(r,r+v_\de+w_\delta),L^k r\rag \,\dd s\nonumber\\&\quad-\int_0^t	 \lag B(v_\de+w_\delta, r), L^k r\rag\, \dd s=I_1+I_2+I_3.
\end{align}To estimate $I_1$,
we integrate by parts  and use     \eqref{1.17}  and  the inequalities of Cauchy--Schwarz and Young: 
\begin{align*}
|I_1|&\le \int_0^t \|\xi_\de\|_{k-1} \|r\|_{k+1}	\,\dd s\\&\le  C\int_0^t \left( \|h\|_{k-1}^2+ \|\nu L v_\de\|_{k-1}^2+ \|B(v_\de)\|_{k-1}^2\right) \dd s+\frac{\nu}{4}\int_0^t  \|r\|_{k+1}^2  \dd s. 
\end{align*}
   By a change of variable,   we have
\begin{align*}
	 \int_0^t  \|h\|_{k-1}^2 \, \dd s&\le \int_0^{T\de}  \|h\|_{k-1}^2  \,\dd s, \\
	  \int_0^t  \| L v_\de\|_{k-1}^2 \, \dd s&\le \de \int_0^T \|v\|_{k+1}^2 \,\dd s,\\
	  \int_0^t   \|B(v_\de)\|_{k-1}^2 \,\dd s &\le \delta \int_0^T   \|B(v)\|_{k-1}^2\, \dd s\le C\delta \int_0^T   \|v\|_k^2\, \dd s, \quad t\in J_{T\de}.
\end{align*}Thus,   there is $\e_\delta=\e_\de(M)>0$  not depending   on $t\in J_{T\de}$ and $u_0\in B_{H^{k+1}}(0,M)$   such that~$\e_\de \to 0$ as~$\delta\to 0^+$ and
\begin{equation}\label{1.19}
|I_1|\le \e_\de +	\frac{\nu}{4}\int_0^t  \|r\|_{k+1}^2 \, \dd s, \quad t< \tilde T^\de.
\end{equation}  
 We estimate $I_2$ and $I_3$ as follows:
\begin{align*}
|I_2|&\le \int_0^t |\lag   B(r,r+v_\de+w_\delta),L^k r\rag |	\,\dd s\\&\le  C\int_0^t \left( \|B(r)\|_{k-1}^2+ |\lag   B(r, v_\de+w_\delta),L^k r\rag | \right) \dd s+\frac{\nu}{4}\int_0^t  \|r\|_{k+1}^2  \dd s,
\\ |I_3|&\le \int_0^t |\lag   B(v_\de+w_\delta,r),L^k r\rag | 	\,\dd s.
\end{align*}
Note that  
 \begin{align*}
  \int_0^t  \|B(r)\|_{k-1}^2  \,\dd s &\le C \int_0^t  \|r\|_k^4 \, \dd s,\\
  \int_0^t   |\lag   B(r, v_\de+w_\delta),L^k r\rag |  \,\dd s &\le  C\int_0^t \left(\|v_\de\|_{k+1} +\|w_\de\|_{k+1} \right)   \| r\|_{k}^2 \, \dd s,\\
  \int_0^t   |\lag   B( v_\de+w_\delta, r),L^k r\rag | \,\dd s &\le  C\int_0^t \left(\|v_\de\|_k +\|w_\de\|_k \right)   \| r\|_{k}^2 \, \dd s,
 \end{align*}
 where we used the inequalities (see \cite{CF1988})
 \begin{align*}
  	|\lag   B(a, b),L^k b\rag |&\le C \|a\|_k \|b\|_k^2,\\
 	|\lag   B(a, b),L^k c\rag |&\le C  \|a\|_k \|b\|_{k+1}\|c\|_k, \quad  a,b \in H^k,\,  c\in H^{k+1}.
 \end{align*} Thus
\begin{align*}
 |I_2+I_3|&\le C\int_0^t \left(\|v_\de\|_{k+1} +\|w_\de\|_{k+1} \right)   \| r\|_{k}^2 \, \dd s\\&\quad+C \int_0^t  \|r\|_k^4 \, \dd s
+	\frac{\nu}{4}\int_0^t  \|r\|_{k+1}^2 \, \dd s. 
\end{align*}
Combining this with  \eqref{1.18} and \eqref{1.19}, we obtain  
 $$
  \|r\|_k^2 \le   \e_\de  + C\int_0^t \left(\|v_\de\|_{k+1} +\|w_\de\|_{k+1}  \right)   \| r\|_{k}^2 \, \dd s+C \int_0^t  \|r\|_k^4 \, \dd s, \quad t<\tilde T^\delta.
 $$By the Gronwall inequality,
 $$
  \|r\|_k^2 \le \left(\e_\de +C \int_0^t  \|r\|_k^4 \, \dd s\right)  \exp\left(C\int_0^t (\|v_\de\|_{k+1} +\|w_\de\|_{k+1} )   \, \dd s \right).  $$
	 For  $\de\le \delta_0(M)$ and $t\in J_{T\delta}$, we have
	$$
	\int_0^t (\|v_\de\|_{k+1} +\|w_\de\|_{k+1} )     \, \dd s=\int_0^{t\de} (\de\|v\|_{k+1} +\|w\|_{k+1} )   \, \dd s\le 1.
	$$Thus
	\begin{equation}\label{1.20}
	 \|r\|_k^2 \le  \e_\de +C \int_0^t  \|r\|_k^4 \, \dd s, \quad t<\tilde T^\delta,
	\end{equation}where  $C=C(M)>0$ does not depend on $t$, $\delta$, and $u_0$, and by the same letter~$\e_\delta$ we denote $e^C \e_\delta$. 
	Let us set
	$$
\Phi(t)=  \e_\de +C \int_0^t \|r\|_k^4\, \dd s . 
$$Inequality   \eqref{1.20} implies that 
$
(\dot \Phi)^{1/2}\le C\,\Phi, 
$ which is equivalent to
 $
 \dot \Phi/\Phi^2\le C.
$ Integrating the latter, we obtain
$$
\Phi(t)\le \e_\delta  (1-  C\e_\delta t)^{-1}, \quad t<\tilde T^\delta.
$$
Choosing $\delta$ sufficiently small, we see that 
$$\Phi(t)\le  2 \e_\delta<1 , \quad t<\tilde T^\delta.
$$This implies both assertions~\eqref{1.13} and \eqref{1.15} and completes the proof of the~proposition.
 \end{proof}
The following is the main result of this section.
  \begin{theorem}\label{T:1.3}
  Let $\HH$ be a   subspace of $H^{k+1}$ such that  Conditions~{\rm(\hyperlink{C1}{\rm C$_1$})} and~{\rm(\hyperlink{C2}{\rm C$_2$})} are satisfied.~Then Eq.~\eqref{1.4} is approximately controllable in   small time by $\HH$-valued controls, i.e.,~for any   $u_0, u_1\in H^{k+1}$   and sufficiently small    $\delta>0$, there is a   control   $\eta_\delta\in \Theta(u_0,h,T\de) \cap L^2(J_{T\delta},\HH)$ such that \begin{equation}\label{1.21}
S_{T\delta}(u_0,h+\eta_\de)\to u_1 \quad\text{in $H^k$ as $\de\to 0^+$}.
\end{equation}
Moreover, the control $\eta_\delta$ can be chosen in the form 
\begin{equation}\label{1.22}
  \eta_\de=R_\de(u_0,u_1)+ \zeta_\delta,
\end{equation}
 where $R_\de: H^k\times H^k\to  L^2(J_{T\delta},\HH)$ is a   linear bounded operator   with a finite-dimensional range and    $\zeta_\delta\in    L^2(J_{T\delta},\HH)$, both   $R_\de$ and $\zeta_\delta$ do   not depend  on~$(u_0,u_1)$. Limit~\eqref{1.21} is
   uniform with respect to $u_0$ and $u_1$ in a
bounded set of $H^{k+1}$.
 \end{theorem}
\begin{proof} Let us denote by
$$
A:H^{k+1}\times L^2(J_T,\HH)\to C(J_T, H^{k+1})\cap W^{1,2}(J_T,H^k), \quad (v_0,g)\mapsto v 
$$  the resolving operator of Eq.~\eqref{1.6} with the initial condition  $v(0)=v_0$, and let~$A_t$ be its  restriction at time~$t$.~By~Condition~{\rm(\hyperlink{C2}{\rm C$_2$})}, the image of the mapping
$$
A_T(0,\cdot):L^2(J_T,\HH)\to H^{k+1}
$$ is dense in $H^{k+1}$.~Hence,  we can construct an approximate right inverse for~$A_T(0,\cdot)$. More precisely, by 
Proposition~2.6~in~\cite{KNS-2018}, for any~$\e>0$, there is a linear bounded operator~$R_\e: H^k\to L^2(J_T,\HH)  $ such~that 
$$
\|A_T(0,R_\e f) -f\|_{k}\le \e \|f\|_{k+1} \quad \text{for    $f\in H^{k+1}$.}
$$ Now let us take any   $M>0$ and any      $u_0, u_1\in B_{H^{k+1}}(0,M)$.~Applying the previous inequality with $f=u_1-A_T(u_0,0)$, we get  
$$
\|A_T(u_0, g_\e)-u_1\|_k \le \e \|u_1-A_T(u_0,0)\|_{k+1}\le \e C,
$$where   $g_\e= R_\e(u_1-A_T(u_0,0))$ and $C=C(M)>0$ is a constant. Combining this with Propositions~\ref{P:1.1} and~\ref{P:1.2}, we complete the proof of the theorem.
 \end{proof}
We close this section with the following result.
 \begin{corollary}\label{C:1.4}
Assume that the conditions of Theorem~\ref{T:1.3} are satisfied.~Then  
Eq.~\eqref{1.4} is approximately controllable in time $T>0$ by $\HH$-valued controls, i.e.,~for any~$\e>0$ and any~$u_0, u_1\in H^k$, there is a control   $\eta\in \Theta(u_1,h,T)\cap L^2(J_T,\HH)$ such that
$$
 \|S_T(u_0,h+\eta)-u_1\|_k<\e. 
$$
\end{corollary}
\begin{proof}
  By the regularising property of   the NS system (see Proposition~\ref{P:1.1}), the~solution  corresponding to   initial point $u_0\in H^k$ and control $\eta=0$ becomes instantaneously   smooth. Combining this with    Theorem~\ref{T:1.3}, we see that Eq.~\eqref{1.4} is approximately controllable in small time in the sense that,   for~any~$u_0, u_1\in H^k$,
 there is a   control   $\tilde\eta_\delta\in \Theta(u_0,h, T\delta) \cap  L^2(J_{T\delta},\HH)$ such that    
\begin{equation}\label{1.23}
S_{T\delta}(u_0,h+\tilde\eta_\delta)\to u_1 \quad\text{in $H^k$ as $\de\to 0^+$}.
\end{equation}
   Thus, to prove approximate controllability in fixed time,  it suffices to show that,  for any $T, \e>0$   and   any~$u_1\in H^k$,    there is  a   control $\eta_1
\in \Theta(u_1,h,T)\cap L^2(J_T,\HH)  $ such~that
$$
\|\RR_T(u_1,h+\eta_1) - u_1 \|_k<\e,
$$  where  the initial condition and the target coincide   with $u_1$.   
  By Proposition~\ref{P:1.1}, there is a time~$\tau>0$   such that the control~$\eta=0$ is in the set~$\Theta(u_1,h,\tau)$ and the function $S_\cdot(u_1,h):J_\tau\to H^k$ is continuous. Taking $\tau$ sufficiently small and using
the   property~(ii) in Proposition~\ref{P:1.1}, we find   a number  $r\in (0,\e)$  such that~$\eta= 0$ belongs to~$\Theta(v,h,\tau)$ for any~$v \in B_{H^k}(u_1,r)$      and
$$
  	\|S_t(v,h)-u_1\|_k<\e,   \quad  t\in J_\tau.  
$$   Thus starting from any  initial point $v \in B_{H^k}(u_1,r)$, the solution corresponding to~$\eta=0$   remains in the ball~$B_{H^k}(u_1,\e)$  on the   time interval $J_\tau$.~If $\tau>T$, then   the proof   is complete.~Otherwise,  applying~\eqref{1.23}   with initial point $u_0'=S_\tau(v,h)$ and target~$u_1$, we find a small time $T'< T-\tau$ and   a   control $ \eta_2\in \Theta(u_0',h,T')\cap L^2(J_{T'},\HH)$ such that 
$$	\|\RR_{T'}(u_0',h+\eta_2) - u_1 \|_k<r.
$$ 
By the choice of $r$ and $\tau$,      if $2\tau+T'>T$, then again  the proof is complete. Otherwise, we   complete the proof   by iterating the above argument finitely many~times.
\end{proof}
\begin{remark}\label{R:1.5}
In this corollary, the control $\eta$ is not of the form \eqref{1.22}.~The affine dependence on $(u_0,u_1)$ is lost after the first application of  zero control in the $r$-neighborhood of~$u_1$. Indeed, this comes from the fact that $S_t(u_1, 0)$ is nonlinear in~$u_1$.~Analysing the   above proof, we easily see that for given~$\e, M>0$ and any~$u_0, u_1\in B_{H^{k+1}}(0,M)$, the restriction~$\eta|_{[0,T\de]}$ of the control     is of the form~\eqref{1.22}, while  the   restriction~$\eta|_{[T\de, T]}$ does not depend~on~$u_0$. 
 \end{remark}


\section{Proof of the Main Theorem}\label{S:2}

The   goal of this section is to show that      Conditions~{\rm(\hyperlink{C1}{\rm C$_1$})} and~{\rm(\hyperlink{C2}{\rm C$_2$})} are verified for different subspaces $\HH$     spanned by  a finite number of    eigenfunctions of the Stokes operator.~Also we prove     the Main Theorem formulated in~the~Introduction.

\subsection{More general formulation}\label{S:2.1}

  Let us fix any numbers $\alpha_i>0$, $i=1, 2, 3$ and   endow the space $\R^3$ with the scalar product 
$$
  \lag x,y \rag_\alpha =\sum_{i=1}^3 \alpha_i^{-1} x_i\, y_i.
$$
  For any~$\ell\in \Z^3_*$, let us  denote  
$$
c_\ell(x)=l(\ell)\cos\lag  \ell, x \rag_\alpha,\quad s_\ell(x) = l(\ell)
\sin\lag  \ell, x \rag_\alpha,
$$
where    $\{l(\ell),l(-\ell)\}$  is any orthonormal basis   in the hyperplane 
$$
\ell^{\bot_\alpha} =\{x\in \R^3:\lag x,\ell\rag_\alpha=0\}.
$$
The family~$\{c_\ell, s_\ell\}_{\ell\in \Z^3_*}$ is a complete   orthogonal system   in $H^k$ composed of eigenfunctions of the Stokes operator.~Let $\KK\subset \Z^3_*$ be a finite symmetric set (i.e.,~$\KK=-\KK$).~We~associate with $\KK$   a non-decreasing sequence of finite-dimensional subspaces~by 
\begin{gather}
\HH_0(\KK)= \lspan \{c_{\ell}, s_{\ell}: \ell\in \KK\},\label{2.1}\\
\HH_i(\KK)= \lspan \{\eta_1 +Q(\eta_2,\xi)  :  \eta_1,\eta_2\in\HH_{i-1}(\KK),\,  \xi\in\HH_0(\KK)\},\quad i\ge1,\label{2.2}
\end{gather}where $Q$ is the bilinear form defined by \eqref{1.7}.
\begin{definition} 
	We say that   $\KK\subset \Z_*^3$ is saturating if the subspace~$\cup_{i=1}^\ty \HH_i(\KK)$ is dense in $H^k$.
\end{definition}
 The following  theorem is proved in the next two subsections.
\begin{theorem}\label{T:2.2}
 Assume that
   $\KK\subset \Z_*^3$  is a saturating set.~Then~Conditions~{\rm(\hyperlink{C1}{\rm C$_1$})} and~{\rm(\hyperlink{C2}{\rm C$_2$})} are satisfied for the subspace~$\HH=\HH_1(\KK)$, and  therefore   the conclusions of Theorem~\ref{T:1.3} and Corollary~\ref{C:1.4} hold.
\end{theorem} 
The following theorem provides a practical way for constructing saturating sets.~Recall that~$\KK\subset \Z^3_*$ is   a generator if any vector of~$\Z^3$
is a finite linear combination of vectors of $\KK$ with integer
coefficients.
\begin{theorem}\label{T:2.3}
	If  a finite symmetric set $\KK\subset \Z^3_*$ is a    generator, then it is saturating.
\end{theorem}
See Section~\ref{S:3} for a proof of this result. Now we turn to the proof of the results formulated in the Introduction.
\begin{proof}[Proof of the Main Theorem and the Corollary]
 For any   $\ell\in \R^3_*$, we denote by~$P_\ell$   the orthogonal projection in $\R^3$ onto the hyperplane   $\ell^{\bot_\alpha}$. Then,  for any $a\in \R^3$, we have the equalities 
$$
\Pi(a\cos\lag\ell,x\rag_\alpha)=(P_\ell a)\cos\lag\ell,x\rag_\alpha,\quad\Pi(a\sin\lag\ell,x\rag_\alpha)=(P_\ell a)\sin\lag\ell,x\rag_\alpha.
$$ These equalities and   some  simple   trigonometric computations show~that 
\begin{align}2Q(a\cos\lag\ell_1,x\rag_\alpha,&\,b\sin\lag\ell_2,x\rag_\alpha)=\cos\lag\ell_1-\ell_2,x\rag_\alpha P_{\ell_1-\ell_2}\left(\lag a,\ell_2\rag_\alpha b-\lag b,\ell_1\rag_\alpha a\right)\nonumber\\&\quad+\cos\lag\ell_1+\ell_2,x\rag_\alpha P_{\ell_1+\ell_2}\left(\lag a,\ell_2\rag_\alpha b+\lag b,\ell_1\rag_\alpha a\right),\label{2.3}\\
2Q(a\cos\lag\ell_1,x\rag_\alpha,&\,b\cos\lag\ell_2,x\rag_\alpha)=\sin\lag\ell_1-\ell_2,x\rag_\alpha P_{\ell_1-\ell_2}\left(\lag a,\ell_2\rag_\alpha b-\lag b,\ell_1\rag_\alpha a\right)\nonumber\\&\quad-\sin\lag\ell_1+\ell_2,x\rag_\alpha P_{\ell_1+\ell_2}\left(\lag a,\ell_2\rag_\alpha b+\lag b,\ell_1\rag_\alpha a\right),\label{2.4}
\\
2Q(a \sin \lag \ell_1,x\rag_\alpha, &\, b\sin \lag \ell_2,x\rag_\alpha) = \sin\lag \ell_1-\ell_2,x \rag_\alpha P_{\ell_1-\ell_2} \left(\lag a,\ell_2\rag_\alpha b- \lag b,\ell_1\rag_\alpha a \right)\nonumber\\&\quad +\sin\lag \ell_1+\ell_2,x \rag_\alpha P_{\ell_1+\ell_2} \left(\lag a,\ell_2\rag_\alpha b+ \lag b,\ell_1\rag_\alpha a \right)\label{2.5}
\end{align}
for any $\ell_1,\ell_2\in \Z^3_*$, $a\in \ell_1^{\bot_\alpha}$, and $b\in \ell_2^{\bot_\alpha}$.  
 	Let us consider the set
$$
\tilde \KK=\{(\pm1,0,0), (0,\pm1,0), (0,0,\pm1)\}
$$  which is, clearly, a generator. Due to identities~\eqref{2.3}-\eqref{2.5}, the subspace  $\HH_1(\tilde\KK)$ is contained in the subspace defined by~\eqref{0.4}.  Applying  Theorems~\ref{T:2.2} and~\ref{T:2.3}~with      the set $\tilde \KK$   and taking $\alpha=(1,1,1)$, 
  we obtain the Main Theorem and the~Corollary.
 \end{proof}

  \begin{remark}\label{R:2.4}
 The papers  \cite{shirikyan-cmp2006, shirikyan-aihp2007, VN-2015} provide a sharp version of the Corollary regarding the dimension of the control space.~In these papers, a nonlinear   saturation property is defined for the 3D NS system~\eqref{1.4}, and  in the case $h\equiv 0$   and $\alpha=(1,1,1)$, the system is proved to be approximately controllable in~time~$T>0$ by $\HH_0(\KK)$-valued controls if and only if $\KK$ is a generator (see Theorem~4.5 in~\cite{VN-2015}).  The subspace $\HH_1(\KK)$ is strictly larger than $\HH_0(\KK)$.
  \end{remark}

 \subsection{Checking  Condition (C$_1$)}\label{S:2.2}
 Let us denote $\HH=\HH_1(\KK)$ and consider 
  the function
\begin{equation}\label{2.6}
 w(t)=\sum_{\ell\in \KK} \left(\psi^c_\ell(t)  \,c_\ell+\psi^s_\ell(t) \, s_\ell\right), 
 \end{equation}
 where 
  $\{\psi_\ell^c,\psi_\ell^s\}_{\ell\in \KK}$ are  any functions in $ W^{1,2}(J_T,\R)$ verifying the boundary conditions  
  $$\psi_\ell^c(0)=\psi_\ell^c(T)=\psi_\ell^s(0)=\psi_\ell^s(T)=0.
  $$ As $ c_\ell$ and $  s_\ell$ are eigenfunctions of the Stokes operator, we have  $Lw(t)\in  \HH$ for~$t\in J_T$.
   Let us denote~$\zeta=\dot w+B(w)$ and show that~$\zeta\in L^2(J_T, \HH)$. Indeed, we have $\dot w \in L^2(J_T, \HH)$ by the  construction. Moreover,   the equality  
$$  B(w)=\sum_{\ell_1,\ell_2\in\KK} Q(\psi^c_{\ell_1}(t)  c_{\ell_1}+\psi^s_{\ell_1}(t)  s_{\ell_1}, \psi^c_{\ell_2}(t)  c_{\ell_2}+\psi^s_{\ell_2}(t)  s_{\ell_2})
$$   implies that   $B(w)\in C(J_T, \HH)$. Thus,  Condition~{\rm(\hyperlink{C1}{\rm C$_1$})} is~satisfied.

  \subsection{Checking  Condition   (C$_2$)}\label{S:2.3}

  Condition~{\rm(\hyperlink{C2}{\rm C$_2$})} is   more subtle  and   is  satisfied under     additional hypotheses on the functions   $\{\psi_\ell^c, \psi_\ell^s\}_{\ell\in \KK}$ entering~\eqref{2.6}.    We shall use the notion of 
  observable family of functions from~\cite{KNS-2018}. 
  
    To prove the approximate controllability of Eq.~\eqref{1.6}, we shall  use some arguments close to the ones   in Section~4 in~\cite{KNS-2018},~where the   controllability of the linearised
       2D NS system is established.        
An important   difference is that there is no    diffusion term in   Eq.~\eqref{1.6}, so we do not have a   parabolic regularisation property.~As a consequence,  we cannot use the $L^2$-adjoint~problem and the~$H^k$-adjoint is of rather complicated form and seems to be ill-suited for our   purposes. Instead of using the adjoint~problem, we use   the equation 
satisfied by the derivative of the resolving operator with respect to the initial condition (see~\eqref{2.13}).

 {\it Step~1.~Observable family.}~A family   $\{\phi_i\}_{i=1}^n\subset L^2(J_T,\R)$ is said to be  observable\footnote{Note that, the      observability property we use  here is     stronger than the one introduced in Definition~4.1 in~\cite{KNS-2018}.} if  for any subinterval $J\subset J_T$, any continuous function~$b:J\to\R$,   and any $C^1$-functions $a_i:J\to\R$   
 	the equality
	\begin{equation} \label{2.7}
		b(t)+\sum_{i=1}^n a_i(t)\phi_i(t) =0\quad\text{in $L^2(J, \R)$}
	\end{equation}
	implies that     $a_i\equiv b\equiv0$, $1\le i\le n$ on $J$.~An example of  observable family can be constructed as follows.~Let
	$\phi_i  :J_T\to \R$ be bounded measurable functions    having
	  left and right limits at any point of $J_T$.~Moreover, let  there be disjoint countable dense sets~$\{\D_i \}_{i=1}^n$ in~$J_T$ such that~$\phi_i$  is discontinuous on~$\D_i$ and  continuous on~$J_T\backslash \D_i$.~Then the family~$\{\phi_i\}_{i=1}^n$ is~observable. Indeed,   take any~$1\le i\le n$ and any $s\in \D_i$. All the functions~$\phi_j$, $j\neq i$ are continuous at~$s$, so the jump at   $s$ of the function on the left-hand side of~\eqref{2.7} is equal to~$a_i(s)(\phi_i(s^+)-\phi_i(s^-)=0$. It follows that  $a_i(s)=0$ for any $s\in \D_i$, hence $a_i\equiv0$ on $J$,  by density and continuity.  By~\eqref{2.7}, we have also $b\equiv 0$ on $J$.

Let us now fix  an observable family of functions $\{\phi_\ell^c, \phi_\ell^s\}_{\ell\in \KK}\subset L^2(J_T,\R)$ and~denote
$$
\psi_\ell^c(t)=\phi(t)\int_0^t\phi_\ell^c(\tau) \,\dd \tau, \quad  \psi_\ell^s(t)=\phi(t)\int_0^t\phi_\ell^s(\tau) \,\dd \tau, \quad t\in J_T, 
$$where $\phi:J_T\to \R $ is  a $C^1$-function such that $\phi(t)=0$ if and only if $t=T$.
  Of~course, Condition~{\rm(\hyperlink{C1}{\rm C$_1$})} remains true in this case.

 {\it Step~2.~Reduction.}  Let us fix any $k\ge3$ and denote by
 $R(t,\tau):H^k\to H^k$,   $0\le \tau\le t\le 1$    the two-parameter resolving operator of the linearised problem 
	\begin{equation} \label{2.8}
		\dot v +Q (w,v)=0, \quad v(\tau)=v_0.
	\end{equation}Then
	$$
	A:L^2(J_T,H^k) \to H^k, \quad g\mapsto \int_0^T R(T,\tau) g(\tau)\,\dd \tau, 
	$$is the resolving operator of Eq.~\eqref{1.6}   with initial condition~$v(0)=0.$
	Denote by~${\mathsf P}_\HH:H^k\to H^k$      the orthogonal projection onto~$\HH$ in~$H^k$. Our goal is to show that the image of the linear operator
	$$
	A_1:L^2(J_T, H^k)\to H^k, \quad   A_1=A{\mathsf P}_\HH
	$$is dense in $H^k$. It is   equivalent to show that the kernel of the adjoint operator
	$$
	A_1^*: H^k \to L^2(J_T, H^k), \quad z\mapsto {\mathsf P}_\HH  R(T,\tau)^* z
	$$ is trivial, where $  R(T,\tau)^*:H^k\to H^k$ is the $H^k$-adjoint of $R(T,\tau).$
	
	{\it Step~3.~Triviality of~$\ker A_1^*$.}~Let us take any $z \in \ker A_1^*$ and show that  $z=0.$ Indeed, for any~$g\in\HH$,  we have
$$
	\left(g,R (T,\tau)^*z\right)_{k}=0\quad \text{for a.e. $\tau\in J_T$}.
$$ This implies that 
\begin{equation}\label{2.9}
	\left(R (T,\tau) g, z\right)_{k}=0\quad \text{for any  $\tau\in J_T$},
\end{equation}
by continuity in $\tau$  of $R (T,\tau) g$. Let fix any $T_1\in (0,T)$ and 
  rewrite this equality as follows:
\begin{equation} \label{2.10}
	\left(R (T_1,\tau) g, z_1\right)_{k}=0\quad \text{for any  $\tau\in J_{T_1}$},
\end{equation}where $z_1=R (T,T_1)^*z$. Taking $\tau=T_1$, we obtain
\begin{equation} \label{2.11}
	\left(g,z_1\right)_{k}=0,
\end{equation}
i.e., $z_1$ is orthogonal to $\HH=\HH_1(\KK)$ in $H^k$. Let us show that $z_1$ is orthogonal also to~$\HH_2(\KK)$. To this end, let us denote
\begin{equation}\label{2.12}
y(t,\tau)=R (\tau+t,\tau) g,
\end{equation} and note that $y(t,\tau)$ is the solution of the problem
\begin{align*}
\dot y(t,\tau)+Q(w(\tau+t),y(t,\tau))&=0, \quad   t\in (0, T-\tau), \\
y(0,\tau)&=g. 
\end{align*}
It follows that  $Y(t,\tau)=\frac{\p }{\p \tau} y(t,\tau)$ is the solution of \begin{align}
\dot Y(t,\tau)+Q(w(\tau+t),Y(t,\tau))+Q(\dot w(\tau+t),y(t,\tau))&=0, \,  t\in (0, T-\tau),\label{2.13}\\
Y(0,\tau)&=0.\label{2.14}
\end{align}
On the other hand, taking the derivative in $\tau$ of~\eqref{2.12} and choosing~$t=T_1-\tau$, we obtain
\begin{align}
	 \frac{\p }{\p \tau} R(T_1,\tau)g&= Y(T_1-\tau,\tau)-\dot R(T_1,\tau)g\nonumber\\&=Y(T_1-\tau,\tau)+ Q(w(T_1),R(T_1,\tau)g).\label{2.15}
\end{align} In the last equality, we used Eq.~\eqref{2.8}. Taking the derivative of~\eqref{2.10} in $\tau$ and using equalities \eqref{2.13}-\eqref{2.15}, we arrive at
\begin{align*}
0&=\int_0^{T_1-\tau} 	 \left(Q(w(\tau+t),Y(t,\tau))+Q(\dot w(\tau+t),y(t,\tau)), z_1\right)_k  \dd t\\
&\quad-(Q(w(T_1),R(T_1,\tau)g), z_1)_k\\
&=\int_0^{T_1-\tau}	 \left(Q(w(\tau+t),Y(t,\tau)), z_1\right)_k  \dd t+\int_\tau^{T_1}	  \left(Q(\dot w(t),R (t,\tau) g), z_1\right)_k  \dd t\\
&\quad-(Q(w(T_1),R(T_1,\tau)g), z_1)_k.
\end{align*}Differentiating this in $\tau$, we get  
$$
b(\tau)+\sum_{\ell\in\KK}\left(a_\ell^c(\tau)\phi_\ell^c(\tau)+a_\ell^s(\tau)\phi_\ell^s(\tau)\right)=0\quad\text{for $\tau\in J_{T_1}$}, 
$$
where 
\begin{align*}
b(\tau)&=	\frac{\p }{\p \tau} \int_0^{T_1-\tau}\!\!	 \left(Q(w(\tau+t),Y(t,\tau)), z_1\right)_k  \dd t  
-  	 \big(  Q(\dot \phi(\tau) \tilde w(\tau),  g), z_1\big)_k   \\&\quad+\int_\tau^{T_1}\!\!	  \big(Q(\dot w(t),\frac{\p }{\p \tau} R (t,\tau) g), z_1\big)_k  \dd t -\frac{\p }{\p \tau}(Q(w(T_1),R(T_1,\tau)g), z_1)_k,
\\a_\ell^c(\tau)&=-  \phi(\tau) 	  \left(Q( c_\ell,  g), z_1\right)_k   , \,\,  a_\ell^s(\tau) =-  \phi(\tau)   \left(Q( s_\ell,  g), z_1\right)_k   ,\\
\tilde w(\tau)&=\sum_{\ell\in \KK} \int_0^\tau\left(\phi^c_\ell(t)  \,c_\ell+\phi^s_\ell(t) \, s_\ell\right) \dd t. 
\end{align*} 
 The functions~$\{a_\ell^c, a_\ell^s\}_{\ell\in\KK}$ are continuously  differentiable and~$b$ is continuous on~$J_{T_1}$.~By observability of~$\{\phi_\ell^c, \phi_\ell^s\}_{\ell\in \KK}$, we have thus  $a_\ell^c\equiv a_\ell^s\equiv0$ on $J_{T_1}$ for any~$\ell\in\KK$. As a consequence,  
$$
	 \left(Q( c_\ell,  g), z_1\right)_k   = \left(Q( s_\ell,  g), z_1\right)_k  =0,  
$$which, combined with~\eqref{2.11}, implies that  $z_1$ is orthogonal to $\HH_2(\KK)$. Recalling the definition of $z_1$, we conclude that  
$$
	 \left(R(T,T_1)g, z \right)_k       =0  	\quad \text {for any $T_1\in (0,T)$ and $g\in \HH_2(\KK)$.}
$$
Denoting $T_1$ by $\tau$, we obtain
\eqref{2.9}, but now for any~$g$ in~$\HH_2(\KK)$. Iterating this argument, we prove \eqref{2.9} for any~$g\in \cup_{i=1}^\ty\HH_i(\KK)$.~Taking $\tau=T$ and using the saturation hypothesis, we get that $z=0$.~This completes the proof of Condition~{\rm(\hyperlink{C2}{\rm C$_2$})} and that of Theorem~\ref{T:2.2}.

 \section{Saturation property}\label{S:3}

 In this section, we prove Theorem~\ref{T:2.3}.~In what follows, we    write~$\ell_1\nparallel \ell_2$ to indicate that  the vectors $\ell_1,  \ell_2\in \R^3$ are non-parallel.   
    
 {\it Step~1.~Reduction.}~Let us  define a   sequence of finite symmetric  sets in $\Z^3$ as~follows:
 $$
 \KK_0=\KK, \quad \KK_j=\KK_{j-1}\cup\{\ell_1+\ell_2: \ell_1\in \KK_{j-1},\, \ell_2\in \KK,\, \ell_1\nparallel \ell_2\}, \quad j\ge1.
 $$
As $\KK$ is a generator,  this sequence is strictly increasing
 and
 \begin{equation}\label{3.1}
 	\cup_{j=1}^\ty  \KK_j=\Z^3.
 \end{equation} 
 Let us assume that we have shown the inclusion 
 \begin{equation}\label{3.2}
 	\HH_i(\KK_j) \subset \HH_{i+3} (\KK_{j-1}) \quad \text{for any $i\ge0,\,  j \ge1$}, 
 \end{equation}where $\HH_i(\KK_j)$ are the subspaces  defined by \eqref{2.1} and  \eqref{2.2} with $\KK=\KK_j$ and $c_0=s_0=0.$ Then \eqref{3.2} implies that
$$
\HH_0(\KK_j) \subset \HH_3 (\KK_{j-1}) \subset  \HH_6 (\KK_{j-2})\subset \ldots \subset    \HH_{3j}(\KK).
$$ Combining this with \eqref{3.1}, we see that the subspace $\cup_{j=1}^\ty \HH_j(\KK)$ is dense in $H^k$, i.e.,~$\KK$ is saturating. Thus we need to prove \eqref{3.2}.
 
  {\it Step~2.~Proof of \eqref{3.2}.}~We first consider    a particular case.
  
  {\it Step~2.1.}~Let us take any $\ell_1\in \KK_{j-1}$ and  $\ell_2\in \KK$ such that $\ell_1\nparallel \ell_2$ and 
    denote by $ \de=\de(\ell_1,\ell_2)$ one of   two unit vectors in $\ell_1^{\bot_\alpha} \cap \ell_2^{\bot_\alpha}$. In this step,  we show that     
        \begin{equation}\label{3.3}
    \de\cos\lag\ell_1+\ell_2,x\rag_\alpha,\,\de\sin\lag\ell_1+\ell_2,x\rag_\alpha\in\HH_{i+1}(\KK_{j-1}).
    \end{equation}
Indeed, by identity~\eqref{2.3}, we have 
\begin{align} 
2Q(b\cos\lag\ell_2,x\rag_\alpha,&\, a\sin\lag\ell_1,x\rag_\alpha)=-\cos\lag\ell_1-\ell_2,x\rag_\alpha P_{\ell_1-\ell_2}\left(\lag a,\ell_2\rag_\alpha b-\lag b,\ell_1\rag_\alpha a\right)\nonumber\\&\quad+\cos\lag\ell_1+\ell_2,x\rag_\alpha P_{\ell_1+\ell_2}\left(\lag a,\ell_2\rag_\alpha b+ \lag b,\ell_1\rag_\alpha a\right)\label{3.4}
\end{align}
for any $a\in \ell_1^{\bot_\alpha}$, and $b\in \ell_2^{\bot_\alpha}$. Summing \eqref{2.3} and \eqref{3.4}, we obtain
\begin{align}\label{3.5}\cos\lag\ell_1+\ell_2,x\rag_\alpha &P_{\ell_1+\ell_2}\left(\lag a,\ell_2\rag_\alpha b+\lag b,\ell_1\rag_\alpha a\right)= 
 Q(a\cos\lag \ell_1,x\rag_\alpha,b\sin\lag\ell_2,x\rag_\alpha)\nonumber\\&\quad\,\,\,+
Q(b\cos\lag\ell_2,x\rag_\alpha,a\sin\lag\ell_1,x\rag_\alpha).
\end{align}We take $a=\de$ and $b$ such that~$\lag b, \ell_1\rag_\alpha=1$. This choice is possible since $\ell_1\nparallel \ell_2$. Then \eqref{3.5} becomes 
$$
  \de\cos\lag\ell_1+\ell_2,x\rag_\alpha
=Q(\de\cos\lag\ell_1,x\rag_\alpha,b\sin\lag\ell_2,x\rag_\alpha)+Q(b\cos\lag\ell_2,x\rag_\alpha,\de\sin\lag\ell_1,x\rag_\alpha).
$$
This implies that     $\de\cos\lag \ell_1+\ell_2,x \rag_\alpha \in  \HH_{i+1}(\KK_{j-1})$. The second inclusion in \eqref{3.3} is proved in a similar way.

 {\it Step~2.2}. Now we prove~\eqref{3.2}. Let us take any    $r\in \KK$  such that the family  $\EE=\{\ell_1,\ell_2,r\}$  is  a  linearly independent. This   is possible since $\KK$ is a generator.  To simplify notation, let us denote
 $$
 (x,y,z)_\EE=x \ell_1+y \ell_2+z r
 $$ for any $x,y,z \in\R$.~Then  the family  $\{ (1,-1,1)_\EE,\,   (-1,1,1)_\EE,\, (1,1,-1)_\EE\}$ is linearly independent, so the intersection of the hyperplanes $(1,-1,1)_\EE^{\bot_\alpha}, $ $ (-1,1,1)_\EE^{\bot_\alpha}$, and~$(1,1,-1)_\EE^{\bot_\alpha}  $ is $\{0\}$. To fix the ideas, let us assume that        
   \begin{equation}\label{3.6}
(1,1,1)_\EE\notin  (1,1,-1)_\EE^{\bot_\alpha},
\end{equation}   the   cases  $(1,1,1)_\EE\notin  (-1,1,1)_\EE^{\bot_\alpha}$ and $(1,1,1)_\EE\notin  (1,-1,1)_\EE^{\bot_\alpha}$ are treated in a similar way. By \eqref{3.3}, we have
    \begin{equation}\label{3.7}
  \de(\ell_1,\ell_2)\cos\lag (1,1,0)_\EE, x\rag_\alpha, \,  \de(\ell_1,\ell_2)\sin\lag (1,1,0)_\EE, x\rag_\alpha\in \HH_{i+1}(\KK_{j-1}).
 \end{equation}  
Now writing  
$$
(1,1,1)_\EE=(1,1,0)_\EE+(0,0,1)_\EE,
$$   recalling that $(0,0,1)_\EE=r\in \KK$, and using \eqref{3.5} and \eqref{3.7}, we obtain 
\begin{align}\label{3.8} 
&\cos\lag (1,1,1)_\EE,x\rag_\alpha P_{(1,1,1)_\EE}\left(\lag\de(\ell_1,\ell_2), (0,0,1)_\EE\rag_\alpha b+\lag b,(1,1,0)_\EE\rag_\alpha\de(\ell_1,\ell_2)\right)\nonumber   \\&\quad=Q(\de(\ell_1,\ell_2)\cos\lag(1,1,0)_\EE,x\rag_\alpha,b\sin\lag (0,0,1)_\EE,x\rag_\alpha)\nonumber\\&\quad\quad +  
Q(b \cos \lag  (0,0,1)_\EE,x\rag_\alpha,\de(\ell_1,\ell_2)\sin\lag(1,1,0)_\EE,x\rag_\alpha)\in \HH_{i+2}(\KK_{j-1}) 
\end{align} for any $b\in (0,0,1)_\EE^{\bot_\alpha}$.~As   the family   $\EE$ is a   linearly independent, we have that~$\lag\de(\ell_1,\ell_2),(0,0,1)_\EE\rag_\alpha\neq0$.  Hence, 
 $$
 \GG=\{\lag\de(\ell_1,\ell_2),(0,0,1)_\EE\rag_\alpha b+\lag b,(1,1,0)_\EE\rag_\alpha\de (\ell_1,\ell_2):b\in(0,0,1)_\EE^{\bot_\alpha}\}
 $$ is a two-dimensional subspace of $\R^3$ contained in $(1,1,-1)_\EE^{\bot_\alpha}$, i.e.,    
$\GG=(1,1,-1)_\EE^{\bot_\alpha}.$ 
 Then \eqref{3.6} implies that  
       $P_{(1,1,1)_\EE} \GG=(1,1,1)_\EE^{\bot_\alpha}$. Combining this with \eqref{3.8}, we derive that     \begin{equation}\label{3.9}
  c_{\pm (1,1,1)_\EE}  \subset  \HH_{i+2}(\KK_{j-1}) .
   \end{equation}In a similar way, one proves that 
   \begin{equation}\label{3.10}
  s_{\pm (1,1,1)_\EE}  \subset  \HH_{i+2}(\KK_{j-1}) .
   \end{equation}
  Applying the result of Step~2.1 to the difference 
$$
(1,1,0)_\EE=(1,1,1)_\EE-(0,0,1)_\EE  
$$    and using \eqref{3.9} and \eqref{3.10}, we obtain 
$$
 \de ((1,1,1)_\EE,(0,0,1)_\EE)\cos\lag (1,1,0)_\EE, x\rag_\alpha\in  \HH_{i+3}(\KK_{j-1}).
$$ Combining this with the fact that  $ \de ((1,1,1)_\EE,(0,0,1)_\EE)\nparallel  \de ( \ell_1,\ell_2)$ and   \eqref{3.3}, we~obtain that~$
  c_{\pm(\ell_1+\ell_2)} \in \HH_{i+3}(\KK_{j-1}).
  $ The proof of  $
  s_{\pm(\ell_1+\ell_2)} \in \HH_{i+3}(\KK_{j-1})
  $ is similar. This   completes the proof of~\eqref{3.2}.

     \addcontentsline{toc}{section}{Bibliography}
 \bibliographystyle{alpha}
\def\cprime{$'$} \def\cprime{$'$}
  \def\polhk#1{\setbox0=\hbox{#1}{\ooalign{\hidewidth
  \lower1.5ex\hbox{`}\hidewidth\crcr\unhbox0}}}
  \def\polhk#1{\setbox0=\hbox{#1}{\ooalign{\hidewidth
  \lower1.5ex\hbox{`}\hidewidth\crcr\unhbox0}}}
  \def\polhk#1{\setbox0=\hbox{#1}{\ooalign{\hidewidth
  \lower1.5ex\hbox{`}\hidewidth\crcr\unhbox0}}} \def\cprime{$'$}
  \def\polhk#1{\setbox0=\hbox{#1}{\ooalign{\hidewidth
  \lower1.5ex\hbox{`}\hidewidth\crcr\unhbox0}}} \def\cprime{$'$}
  \def\cprime{$'$} \def\cprime{$'$} \def\cprime{$'$}

\end{document}